\documentclass[12pt]{amsart}
\usepackage{amssymb}
\usepackage{amsmath}
\usepackage{amsthm} 
\usepackage{graphicx,subcaption}    
\usepackage[utf8]{inputenc}
\usepackage{amsfonts}
\usepackage{array}
\usepackage{graphicx}
\usepackage{hyperref}
\usepackage{xcolor}
\hypersetup{
	colorlinks   = true, 
	urlcolor     = blue, 
	linkcolor    = blue, 
	citecolor   = blue 
}

\numberwithin{equation}{section}
\newtheorem{theorem}{Theorem}[section]

\newtheorem{lemma}[theorem]{Lemma}

 \newtheorem{remark}[theorem]{Remark}
\newtheorem{corollary}[theorem]{Corollary} 
\newcommand{\R}{{\mathbb R}}  \newcommand{\Z}{{\mathbb Z}} \newcommand{\N}{{\mathbb N}}

\newcommand{\Cc}{{\mathbb C}}

 \author[Ulanovskii]{Alexander Ulanovskii}
 \address{Department of Mathematics and Physics, University of Stavanger, 4036 Stavanger, Norway}
 \email{alexander.ulanovskii@uis.no}

\author[Zlotnikov]{Ilya Zlotnikov}
\address{Department of Mathematical Sciences, Norwegian University of Science and Technology (NTNU), 7491 Trondheim, Norway} 
\address{
Erwin Schrödinger International Institute for Mathematics and Physics (ESI),
University of Vienna,
Vienna, Austria
}
\email{ilia.k.zlotnikov@ntnu.no}
\title[Periodic Non-uniqueness Sets for Shift-invariant Spaces]{Periodic Non-uniqueness Sets for Shift-invariant Spaces and Parity-Based Obstructions to the Frame Property for  Gabor Systems}

\begin{document}
\begin{abstract}

The goal of this note is twofold. First, we provide explicit examples of periodic (though not necessarily lattice) sets that give rise to Gabor systems failing to form frames. Our constructions depend only on the parity of the window function $g$.

Second, for a wide range of finite-dimensional function spaces $V$ we show that $V$ contains a function $g$ such that a lattice of high density fails to generate a Gabor frame. In particular, we prove that the Gröchenig--Lyubarskii theorem is sharp in the finite-dimensional space of polynomials with Gaussian weight. More precisely, for every $N\in\N$ and every $\alpha,\beta>0$ satisfying $\alpha\beta=\frac{1}{N+1}$, we give an explicit algorithm for finding an even or odd polynomial $p$ of degree at most $N$ such that
$
\mathcal{G}(p(x)e^{-\pi x^2}, \alpha\mathbb{Z} \times \beta\mathbb{Z})
$
does not form a frame. The proofs are constructive, elementary, and based on linear algebra.

\end{abstract}

 \date{}\maketitle

\section{Introduction}

The fundamental problem in Gabor analysis is to determine what assumptions should be imposed on the function $g$ and discrete sets $\Lambda, \Gamma$ to ensure that the Gabor system
$$
\mathcal{G}(g, \Lambda \times \Gamma) = \left\{g(x-\lambda) e^{2\pi i \gamma x}, \quad \lambda \in \Lambda, \,\gamma \in \Gamma \right\}
$$
constitutes a frame for $L^2(\R).$

Formulated at this level of generality, the problem is known to be highly challenging even for a fixed function $g$. A complete answer is currently available only for the Gaussian kernel $g(x) = e^{-ax^2}, a >0$ (see \cite{MR1188007, MR1173117}), the Cauchy kernel $\frac{1}{x+i}$ (see \cite{MR4345947}), and the one-sided exponential $e^{-x}\chi_{[0, +\infty)}(x)$ (see \cite{BELOV2026101892}).

Imposing a lattice structure on at least one of the sets $\Lambda$ and $\Gamma$ places the problem in a significantly more accessible setting, where several complete and partial characterizations are available.

For a given function $g$ we define its frame set by $$\mathcal{F}(g) = \left\{(\alpha,\beta) \in \R^2_{>0} : \quad \mathcal{G}(g, \alpha \Z \times \beta \Z) \text{ is a frame for }L^2(\R)\right\}. $$

For example, the frame set has been completely characterized for totally positive functions, rational functions with poles in a common half-plane, indicator functions of intervals, one-sided exponentially decaying functions, and the first Haar function, see \cite{BELOV2026101892,MR4542702, MR3545108,  MR4722277, MR3763405, MR3053565, MR1955931}.

However, for many important generators, e.g., Hermite functions, B-splines, ratios of exponential polynomials, the  structure of the frame sets is more complicated, and only some sufficient conditions are available, see \cite{MR2529475, MR4832036, MR4865221}. These conditions allow us to identify some region $\mathcal{S}$ which is contained entirely inside the frame set $\mathcal{F}(g).$
In particular, for the $N$-th Hermite function or more generally, for polynomials of degree at most $N$ multiplied by a Gaussian weight, in \cite{MR2529475} it was proved that one can take 
$$
\mathcal{S}:= \left\{(\alpha,\beta)\in\R^2_{>0} : \alpha\beta < \frac{1}{N+1}\right\}.
$$

On the other hand, for every function $g$ the trivial necessary condition for $\mathcal{G}(g, \alpha\Z \times \beta \Z)$ to form a frame for $L^2(\R)$ is $\alpha \beta <1.$

Between these two regions, frame properties of  Gabor systems have been extensively studied for many families of generators, see e.g. \cite{ MR4887696, MR4793698, MR4849800, MR3027914}.
In particular, there is a series of results that indicate obstructions for lattices to generate a Gabor system possessing the frame property for Hermite functions and B-splines. 

For instance, using properties of the Zibulski--Zeevi matrix representation, in  \cite{MR3027914} it was proved that for every odd $g$ satisfying mild decay assumptions and $m \in \N,$ 
\begin{equation}
    \text{if } \alpha\beta = \frac{m}{m+1} \quad \text{then} \quad \mathcal{G}(g, \alpha \Z \times \beta \Z)\,\, \text{does not form a frame for} \,\, L^2(\R).
\end{equation}

In the present note, we continue our investigation of the obstructions to the frame property for the Gabor systems $\mathcal{G}(g, \alpha \Z \times \beta \Z)$ 
and $\mathcal{G}(g, \Lambda \times \Z)$ generated by regular or semi-regular sets. Our approach relies on the classical correspondence (see \cite{jan95}) between Gabor frames and the uniqueness property in shift‑invariant spaces.

Our results fall into two categories.
First, for a fixed even or odd generator $g$ satisfying mild decay assumptions, we construct explicit examples of the sets $\Lambda$ (including, in particular, certain lattices) of large density for which the Gabor system $\mathcal{G}(g, \Lambda \times \Z)$ does not form a frame, see Theorem~\ref{non_uniq_approx_lattice} and Corollary~\ref{cor:Gabor_even_odd} below.

Second, we show that for a collection of windows $g_1,\ldots,g_N$ satisfying suitable decay, linear-independence, and parity assumptions, one can construct periodic sets which are not uniqueness sets for $V^\infty(g_1,\ldots,g_N)$ and have density $2N$ in the odd case and $2N-1$ in the even case, see Theorem~3.5. Moreover, these sets already fail to be uniqueness sets for $V^\infty(g)$ for a suitable non-trivial function $g\in \operatorname{Span}(g_1,\ldots,g_N)$, see Corollary~3.6.

In addition, in Corollary~\ref{cor:sharp_hermite}, we show that for every $N\in\N$ and every $\alpha,\beta,a\in\R_{>0}$ satisfying $\alpha\beta=\frac{1}{N+1}$, there exists a polynomial $p=p(\alpha,\beta,a)$ such that the Gabor system
$\mathcal{G}(g,\alpha\Z\times\beta\Z)$
does not form a frame for $L^2(\R)$, where
$g(x)=p(x)e^{-a x^2}.$
Furthermore, the coefficients of $p$ can be computed explicitly, or approximated numerically with arbitrary precision, by solving a finite system of linear equations. In addition, the polynomial $p$ may be chosen to be even or odd according to the parity of $N$.
In particular, this shows that the boundary for the frame region obtained in~\cite{MR2529475} is sharp in the class of polynomials of degree at most $N$ multiplied by a Gaussian.

The paper is organized as follows. In Sec.~2 we recall the definition and some properties of shift-invariant spaces. In Sec.~3 we present our main results -- Theorems~\ref{non_uniq_approx_lattice}, \ref{t} and Corollaries~\ref{cor:Gabor_even_odd}, \ref{cor:AN_g}, \ref{cor:sharp_hermite}.
The proofs of these results are given in Sec.~4 and Sec.~5. 

\section{Preliminaries}
Throughout this paper we will assume that the generator function belongs to the continuous Wiener amalgam space $W_0$ which consists of continuous functions $g: \R \to \Cc,$ satisfying
\begin{equation}\label{wiener}\|g\|_{W_0} := \sum\limits_{k \in \Z} \|g\|_{L^\infty (k,k+1)} < \infty.\end{equation}
Given $g \in W_0$ and $1\leq p\leq \infty$, the shift-invariant space $V^p(g)$ is defined by 
$$
V^p(g)=\{f(x)=\sum\limits_{n\in\Z} c_n g(x-n), \quad \{c_n\} \in l^p(\Z)\}.
$$

In what follows, we always assume that the function $g$ has stable $\Z$-shifts, i.e. 
for every $p \in [1, \infty]$ there exist positive constants $A,B$ such that 
\begin{equation}\label{eq:stability}
  A\|\{c_n\}\|_{l^p} \leq \|f\|_{L^p(\R)} \leq B\|\{c_n\}\|_{l^p}, \quad \text{for every } f\in V^p(g).  
\end{equation}

This fundamental condition ensures that $V^p(g)$ is a closed subset of $L^p(\R).$ Let us briefly recall some facts concerning the stability of integer shifts, for more information and some generalizations we refer the reader to \cite{MR3336091, Jia1991, MR1850960, MR4950401}.
First, if~\eqref{eq:stability} is satisfied for some $p_0 \in [1, \infty]$ then it is also satisfied for all $p \in [1, \infty]$ simultaneously, see \cite{MR3336091}.
Second, stability of $\Z$-shifts of $g$ is equivalent to the assumption that there is no $x\in[0,1)$ such that
\begin{equation}\label{eq:stability_fourier}
    \hat{g}(x+n) = 0, \quad\text{for every } n\in \Z,
\end{equation}
where $\hat{g}$ stands for the Fourier transform of $g$ defined by
$$
\hat{g}(t) = \int\limits_{\R} g(x) e^{-2\pi i xt}\, dt.
$$

A set $\Lambda$ is called a uniqueness set (US) for the space $V^p(g)$ if 
$$
f(\lambda) = 0 \,\, \text{for every } \lambda \in \Lambda \quad \text{implies }\quad f\equiv 0.
$$

To prove that a given Gabor system does not form a frame, we use the following connection between Gabor systems and shift-invariant spaces, see e.g., \cite{MR4047939}.

\begin{lemma}\label{lem:gabor_frames_uniquenes_sets}
    Assume that $g \in W_0$ has stable integer shifts. Then if for some $x\in[0,1)$ the set $\Lambda +x$ is not a US for $V^\infty(g)$ then $\mathcal{G}(g, \Lambda \times \Z)$ does not form a frame for $L^2(\R).$
\end{lemma}

Below, we also consider a natural extension of shift-invariant spaces to the case of multiple generators.

Let $g_1,...,g_N \in W_0$ with $ N>1.$ The shift-invariant space $V^\infty(g_1,...,g_N)$  is defined by
$$
V^\infty(g_1,\dots,g_N)
=
\left\{
 f(x)=\sum_{j=1}^N\sum_{n\in\mathbb Z} c_{j,n}g_j(x-n),
\, c_j=\{c_{j,n}\}_{n\in\mathbb Z}\in \ell^\infty(\mathbb Z),
\ 1\leq j\leq N
\right\}.
$$
Such spaces play a fundamental role in signal processing, approximation theory, and wavelet analysis.

In what follows, we assume that each generator $g_j$, $j=1,\dots,N$, satisfies the stability condition~\eqref{eq:stability}. Moreover, we assume that the subspaces $V^{\infty}(g_j)$ are mutually independent in the sense that every element $f\in V^\infty(g_1,\dots,g_N)$ admits a unique decomposition of the form
$$
f(x)=\sum_{j=1}^N f_j(x),
\qquad
f_j\in V^\infty(g_j).
$$

The following lemma provides a sufficient condition for this property; see \cite{MR1850960}.

\begin{lemma}\label{lem:mutual_stable}
Assume that $g_1,\dots,g_N\in W_0$. Then the subspaces
$$
V^\infty(g_j), \qquad j=1,\dots,N,
$$
are mutually independent provided that, for every $\xi\in[0,1]$, the vectors
$$
\left(\widehat{g_k}(\xi+m)\right)_{m\in\Z},
\qquad k=1,\dots,N,
$$
are linearly independent.
\end{lemma}

For brevity, we will say that the generators $g_1,\dots,g_N$ satisfy the mutual independence condition if all of the above assumptions are satisfied.

\section{Main Results}
\subsection{Non-uniqueness sets for even and odd generators}
Below, by $\lfloor s \rfloor$ we denote the greatest integer that is less than or equal to $s\in \R.$

 Given any $n\in \N$ and any points $0<x_1<...<x_{n} <n$, we denote by $F_{n}$ the symmetric set
\begin{equation}\label{eq:set_F_def}
   F_n=F_n(x_1,...,x_n):=\{\pm x_1,...,\pm x_n\}=\{x_1,...,x_n\}\cup \{-x_1,...,-x_n\}. 
\end{equation}
We also set $F_0:=\emptyset.$

Let $k\in\mathbb N$. We now introduce three sets
\begin{equation}\label{eq:set_Lambda_def}
\Lambda_k:=\left(F_{\lfloor\frac{k}{2}\rfloor} \cup \left\{0,\frac{k}{2}\right\}\right)+k\Z;
\end{equation}
\begin{equation}\label{eq:set_Gamma_def}
\Gamma_k:= \left(F_{(k-1)/2}\cup \left\{\frac{k}{2}\right\}\right)+k\Z,\quad k\in 2\N-1; 
\end{equation}
\begin{equation}\label{eq:set_Theta_def}
\Theta_k:=F_{k/2}+k\Z,\quad k\in 2\N.
\end{equation}

Observe that every set above is $k$-periodic, and that $\#(\Lambda_k\cap [0,k))=2 \lfloor k/2 \rfloor+2$ and $\#(\Gamma_k\cap [0,k))= \#(\Theta_k\cap [0,k))=k$.

\begin{theorem}\label{non_uniq_approx_lattice}
  Assume that $g\in W_0$ has stable integer shifts.
  Let $ k\in \N.$
  
$(i)$ If $g$ is odd, then no set $\Lambda_k$   \text{is a US for } $V^{\infty}(g).$
      
$(ii)$ Assume that $g$ is even. If $k$ is odd, then  no set $\Gamma_k$   \text{is a US for } $V^{\infty}(g).$ 
      If $k$ is even, then  no set $\Theta_k$   \text{is a US for } $V^{\infty}(g).$

\end{theorem}

Combining this with Lemma~\ref{lem:gabor_frames_uniquenes_sets}, we obtain the following result for Gabor systems. We formulate it in the setting of Theorem~\ref{non_uniq_approx_lattice}.
\begin{corollary}\label{cor:Gabor_even_odd}Assume that $g\in W_0$ has stable integer shifts, and let $k\in \N.$
\begin{enumerate}
    \item[(i)] If $g$ is odd, then  $\mathcal{G}(g, \Lambda_k\times\Z)$ is not a frame for $L^2(\R)$ for any $\Lambda_k$;    
    \item[(ii)]  If $g$ is even, then  $\mathcal{G}(g, \Gamma_k\times\Z)$ is not a frame for $L^2(\R)$ for $ k \in 2\N-1,$ and $\mathcal{G}(g, \Theta_k\times\Z)$ is not a frame for $L^2(\R)$ for $ k \in 2\N$.
\end{enumerate}
\end{corollary}

As a special case, we recover one of the results from \cite{MR3027914}:
 \begin{corollary}\label{cor:lyub_nes}
 Given $\alpha, \beta > 0$. Let $g \in W_0$ be an odd function with stable $\beta\Z$-shifts and $m\in\N$.
 If $\alpha\beta=\frac{m}{m+1}$ then the system $\mathcal{G}(g, \alpha \Z \times \beta \Z)$  does not form a frame for $L^2(\R).$
 \end{corollary}
 \begin{proof}
     Note that it suffices to prove the corollary for the system $\mathcal{G}(g_{\beta}, \alpha \beta \Z \times \Z)$, where $g_{\beta}(x):=g(\beta^{-1}x)$ is an odd function. 

     Note that $g_{\beta}$ has stable $\Z$-shifts.
     By Lemma~\ref{lem:gabor_frames_uniquenes_sets}, we are done if we can show that $ \frac{m}{m+1}\Z$ is not a US for $V^{\infty}(g_{\beta})$.
     
     Assume that $m\in 2\Z$. We use the definition of $\Lambda_k$ in \eqref{eq:set_Lambda_def} where $k=m$ and $$F_{m/2}:=\left\{\pm j\frac{m}{m+1}: j=1,\dots,m/2\right\}.$$ Clearly, $$\Lambda_m=(\{0,m/2\}\cup F_{m/2})+m\Z \supset \frac{m}{m+1}\Z.$$ Now, the result\footnote{In fact, in this case ($m\in 2\Z$), we even get slightly more: $m(\Z+1/2) \cup \frac{m}{m+1}\Z$ is not a US for $V^\infty(g).$} follows from Lemma~\ref{lem:gabor_frames_uniquenes_sets} and  Theorem~\ref{non_uniq_approx_lattice} (i).

If $m\in 2\Z+1,$ we set $k:=2m\in 2\Z$ and  $F_{k/2}:=\{\pm j\frac{m}{m+1}\}_{j=1}^m.$ It is easy to check that $$\left(\{0, k/2\} \cup F_{k/2} )\right) +k\Z = \left\{j\frac{m}{m+1}\right\}_{j=0}^{2m+1}+2m\Z=\frac{m}{m+1}\Z.$$By Theorem~\ref{non_uniq_approx_lattice} (i),  $\frac{m}{m+1}\Z$ is not a US for $V^\infty(g).$
\end{proof}

\subsection{ Non-uniqueness sets for shift-invariant spaces with multiple generators}

In \cite[Theorem 5]{MR1756138}, it was shown, in particular, that a necessary condition for a set $\Lambda$ to be a sampling set for $V^{\infty}(g_1,\dots,g_N)$ is that its density be at least $N$. 
Moreover, for every $\varepsilon>0$ there exists a lattice of density $N-\varepsilon$ that is not a US for $V^\infty(g_1,\dots,g_N)$. 

Below, we show that if the windows $g_1,\dots,g_N$ are all even or all odd, then one can find a lattice whose density is almost twice the critical density and which still fails to be  a uniqueness set for $V^{\infty}(g_1,\dots,g_N)$. More precisely, the density can be taken to be $2N-1$ in the even case and $2N$ in the odd case.

For any $N\in\N, N\geq 2,$ and any points $0<x_1< \dots <x_{N-1}<\frac{1}{2},$ let us introduce two sets: $$A_N:=A_N(x_1,\dots,x_{N-1}):=\{\pm x_1,\dots,\pm x_{N-1}\}\cup\left\{0,\frac{1}{2}\right\}+\Z;$$$$B_N:=\{\pm x_1,\dots,\pm x_{N-1}\}\cup\left\{\frac{1}{2}\right\}+\Z.$$

\begin{remark}\label{rem:lattice}
    Note that, for a suitable choice of the points $(x_1,\dots,x_{N-1})$, the sets $A_N$ and $B_N$ become lattices:
    $$A_N = \frac{\Z}{2N}, \quad x_j = \frac{j}{2N}, \,\, j=1,\dots, N-1,$$
    $$B_N = \frac{\Z+\frac{1}{2}}{2N-1}, \quad x_j = \frac{1}{4N-2}+\frac{j-1}{2N-1}, \,\, j=1,\dots, N-1.$$
\end{remark}

Our first result is as follows.

\begin{theorem}\label{t} Let  $g_1,...,g_N\in W_0, \, N\geq 2,$ satisfy the mutual independence condition.

$(i)$ If every generator $g_j$ is odd, then no set $A_N$ is a uniqueness set for $V^\infty(g_1,\dots,g_N)$.

$(ii)$ If every generator $g_j$ is even, then no set $B_N$ is a uniqueness set for $V^\infty(g_1,\dots,g_N)$.\end{theorem}
Theorem  \ref{t} yields the following corollary, showing that the sets $A_N$  and $B_N$  already fail to be uniqueness sets for a suitably chosen generator in 
${\rm Span}(g_1\dots,g_N)$.

\begin{corollary} \label{cor:AN_g}
    Let $g_1,\dots,g_N, N\geq 2,$ satisfy the assumptions of Theorem~\ref{t}. 
    
$(i)$ If every $g_j$ is odd, then for every set $A_N$ there is a non-trivial generator $g\in {\rm Span}(g_1,\dots,g_N)$ such that $A_N$ is not a US for $V^\infty(g)$. 

    $(ii)$ If every $g_j$ is even, then for every set $B_N$ there is a non-trivial generator $g\in {\rm Span}(g_1,\dots,g_N)$ such that $B_N$ is not a US for $V^\infty(g)$. 
\end{corollary}

Using Lemma~\ref{lem:gabor_frames_uniquenes_sets}, we immediately get that the Gabor systems $\mathcal{G}(g, A_N\times\Z)$ and $\mathcal{G}(g, B_N\times\Z)$ do not form  frames for $L^2(\R)$ for windows $g$ from conditions $(i)$ and $(ii)$ of Corollary~\ref{cor:AN_g} respectively.

We now return to the result of Gr\"{o}chenig and Lyubarskii on Gabor systems generated by polynomials with Gaussian weight, mentioned in the introduction. In \cite{MR2529475}, it was proved that the condition
$
\alpha\beta<\frac{1}{N+1}
$
guarantees that the Gabor system
$
\mathcal{G}(p(x)e^{-a x^2},\alpha\Z\times\beta\Z)
$
forms a frame whenever $p$ is a polynomial of degree at most $N$ and $a>0$.

Below we show that this result is sharp within the class of polynomials of degree at most $N$ multiplied by an arbitrary Gaussian weight $e^{-a x^2}$, $a>0$.

\begin{corollary}\label{cor:sharp_hermite}
For every $a,\alpha, \beta>0$ and $N\in \N$ with $ \alpha \beta = \frac{1}{N+1},$ there exists a polynomial $p_N=p_N(a,\alpha,\beta)$ of degree at most $N$ such that the Gabor system $\mathcal{G}(p_N(x)e^{-ax^2}, \alpha\Z \times \beta\Z)$
 fails to form a frame in $ L^2(\R).$ Furthermore, $p_N$ can be chosen to be even or odd according to the parity of $N$.
\end{corollary}
\begin{proof}
Fix $N\in \N, N \geq 2$ and denote by $\mathcal{P}_{N,a}$ the space of polynomials of degree at most $N$ multiplied by a Gaussian weight $e^{-ax^2}$.

    By substitution of variables, we are done if for every $a>0$ there is a $g\in\mathcal{P}_{N,a}$
    such that the system $\mathcal{G}(g, \frac{1}{N+1}\Z \times \Z)$ does not form a frame in $L^2(\R)$. By Lemma~\ref{lem:gabor_frames_uniquenes_sets}, it suffices to find a $g \in \mathcal{P}_{N,a}$ such that for some $s\in[0,1)$ the set $\frac{\Z}{N+1} + s$ is not a US for $V^{\infty}(g)$.

For $N=1$, the result is already known for $p(x)=x$, see~\cite{MR3027914}.
It also follows from Corollary~\ref{cor:lyub_nes} applied with $m=1$ to the odd window
$g(x)=xe^{-ax^2}$. Thus, in what
follows, we may assume that $N\geq 2$.

We will distinguish the cases of odd and even $N$.
    
    Suppose that $N$ is odd.
    Observe that the functions
    $$
    g_k(x) = x^{2k-1} e^{-ax^2} \in W_0, \quad k=1,\dots,\frac{N+1}{2},
    $$
    are odd and mutually independent. Indeed, the latter follows from Lemma~\ref{lem:mutual_stable}, since it is easy to check that the vectors
    $$ \{\hat g_k(m+\xi)\}_{m \in \Z}, \quad k=1,\dots, \frac{N+1}{2},$$ are linearly independent for every $\xi \in[0,1]$, since the $\sum_{k=1}^\frac{N+1}{2} a_k \hat g_k$ is a non-trivial polynomial with Gaussian weight and it cannot vanish on arithmetic progression.
    
    Choosing the nodes as in Remark~\ref{rem:lattice}, from part (i) of Corollary~\ref{cor:AN_g}, we deduce that there exist a non-trivial $g\in {\rm Span} (g_1,\dots, g_{\frac{N+1}{2}}) \subset \mathcal{P}_{N,a}$ such that $$A_{\frac{N+1}{2}} = \frac{\Z}{N+1} \text{ is not a US for } V^{\infty}(g).$$ This finishes the proof of the case $N\in 2\N-1.$

    The proof for even $N$ is somewhat similar: one can show that the space $\mathcal{P}_{N,a}$ contains $N/2 + 1$ mutually independent even functions
    $$
    g_k(x) = x^{2k-2} e^{-ax^2} \in W_0, \quad k=1,\dots, N/2+1.
    $$ 
    Then, it follows from part (ii) of Corollary~\ref{cor:AN_g} and Remark~\ref{rem:lattice} that the set $$B_{\frac{N}{2}+1}=\frac{\Z + \frac{1}{2}}{2(\frac{N}{2}+1) - 1} = \frac{\Z }{N+1} + \frac{1}{2(N+1)}$$ is not a US for $V^{\infty}(g)$ for some $g \in {\rm Span} (g_1, \dots, g_{\frac{N}{2}+1}) \subset \mathcal{P}_{N,a}$. This finishes the proof of the corollary.
\end{proof}

\section{Proof of Theorem~\ref{non_uniq_approx_lattice}}
\subsection{Odd generators}

Let us start with proving Theorem~\ref{non_uniq_approx_lattice} for generators $g$ satisfying $g(x)=-g(-x), \, x\in \R.  $ 
Fix some $k\in \N.$

We introduce auxiliary function
$$f_k(x):=\sum_{n\in\Z}g(x-kn)\in V^\infty(g).$$
In the next lemma we collect several properties of the function $f_k$.  

\begin{lemma}\label{lem:f_k_properties}
    For every $x\in \R$ the following equations are true:
\begin{enumerate}
    \item[$(A_1)$]  $f_k(x+k) = f_k(x),$
    \item[$(A_2)$] $ f_k(-x)=-f_k(x),$
    \item[$(A_3)$] $f_k(-x+k/2)= -f_k(x+k/2),$
    \item[$(A_4)$] $f_k(kn)=f_k(k/2+kn)=0, \, \, \text{ for every }n\in\Z.$
    \item[$(A_5)$] The functions $f_k(x), f_k(x-1),\dots,f_k(x-k+1)$ are linearly independent.
\end{enumerate}
\end{lemma}
\begin{proof}[Proof of the Lemma]
Relations $(A_1)$ and $(A_2)$ are obvious.
   Let us check $(A_3)$. Indeed, since $f_k$ is odd and $k$-periodic, 
$$f_k(-x+k/2)=-f_k(x-k/2)=-f_k(x-k/2+k)=-f_k(x+k/2).$$ Plugging $x = 0$ in $(A_2),(A_3)$ and using $(A_1)$, we get $(A_4).$ 

Finally, any non-trivial linear combination $f$ of the functions $f_k(\cdot-j)$ can be written as
$$f(x)=\alpha_1f_k(x)+\dots+\alpha_kf_k(x-k+1)=\sum_{n\in\Z}c_ng(x-n),\quad c_n=\alpha_j, n=km+j, m\in\Z.$$Therefore,  $f$ is a non-trivial element of $V^\infty(g)$, and 
condition $(A_5)$ follows from~\eqref{eq:stability} for $p=\infty$.
\end{proof}

Condition $(A_5)$ of the lemma above immediately yields the following corollary, which will be useful in what follows.
\begin{corollary}\label{claim:linear_indep}
\begin{enumerate} 
    \item[$(i)$] If $k$ is even then the family of functions
$$\mathcal{M}_e:=\{f_{k}(\,\cdot\,),f_k( \,\cdot-k/2)\}\cup\{f_k(\,\cdot-j)+f_k(\,\cdot-k+j)\}_{j=1}^{k/2-1}$$
is linearly independent.
    \item[$(ii)$]  If $k$ is odd then the family of functions
$$\mathcal{M}_o:=\{f_{k}(\,\cdot\,)\}\cup\{f_k(\,\cdot-j)+f_k(\,\cdot-k+j)\}_{j=1}^{(k-1)/2}$$
is linearly independent.
\end{enumerate}
\end{corollary}

Now, we are ready to pass to the proof of the main theorem. 
\begin{proof}[Proof of Theorem~\ref{non_uniq_approx_lattice}, part $(i)$]
Recall that we are given some $k \in \N$, the sets $\Lambda_k$ and $F_{\lfloor\frac{k}{2}\rfloor}$ defined in~\eqref{eq:set_F_def} and~\eqref{eq:set_Lambda_def}. 
There are two possibilities:

{\bf Case $k \in 2\N$.} Then $$F_{k/2}  =\{\pm x_1, \pm x_2, \dots, \pm x_{k/2}\} \subset \left(-\frac{k}{2},\frac{k}{2}\right) \setminus \{0\}.$$
To prove that $\Lambda_k$ is not a US for $V^{\infty}(g)$, it suffices to find a $k$-periodic function $f$ from $V^{\infty}(g)$ such that
$$f(0) = f(k/2) = f(x_j) = f(k-x_j) = 0 \quad \text{ for every } j = 1, \dots, k/2.$$

Note that the cardinality of the set $\mathcal{M}_e$ from Lemma~\ref{claim:linear_indep} is equal to $k/2+1.$
Therefore, we can solve the system of linear equations and find $k/2+1$ coefficients $a_j$ such that the function 
\begin{equation}\label{eq:f_def_aj}
  f(x)=a_0 f_k(x)+\sum_{j=1}^{k/2-1}a_j(f_k(x-j)+f_k(x-k+j))+a_{k/2}f_{k}(x-k/2),  
\end{equation}
is not identically zero, belongs to $V^{\infty}(g),$ and satisfies
\begin{equation}\label{eq:f_vanishes_F}
    f(x_j)=0, \quad j=1,...,k/2.
\end{equation}

By equation $(A_1)$ from Lemma~\ref{lem:f_k_properties},  we have \begin{equation}\label{eq:f_is_k_periodic}
 f(x+k)\equiv f(x).   
\end{equation} 
Using $(A_1)$ and $(A_2)$, we observe that
$$f_k(-x-j)+f_k(-x-k+j)=-(f_k(x+j)+f_k(x+k-j))=-(f_k(x-k+j)+f_k(x-j)).$$ Together with $(A_2)$ and $(A_3)$, this implies $f(-x)=-f(x),$ whence
\begin{equation}
    f(0)=f(k/2)=0, \quad \text{and } \quad  \quad f(k-x_j) =f(-x_j) = - f(x_j) = 0. 
\end{equation}
This finishes the proof of the case $k \in 2\Z.$

{\bf Case $k \in 2\N+1$.}  It suffices to find a $k$-periodic function $f$ from $V^\infty(g)$ that vanishes on $\{0, k/2\} \cup F_{\frac{k-1}{2}}.$ Similarly to the previous case, since the cardinality of $\mathcal{M}_o$ is equal to $(k+1)/2$, one can find non-trivial coefficients $a_0, \dots, a_{(k-1)/2}$ such that the function 
$$f(x)=a_0 f_k(x)+\sum_{j=1}^{(k-1)/2}a_j(f_k(x-j)+f_k(x-k+j))$$ vanishes on $\{x_1, \dots, x_{\frac{k-1}{2}}\}.$ Clearly, $f\in V^{\infty}(g).$
Using equations $(A_1)-(A_4)$ from Lemma~\ref{lem:f_k_properties}, one can check that $f(0)=f(k/2)=0$ and $f(k-x)=-f(x).$
We leave the details to the reader.

Finally, we note that if $k=1$, then $F_0=\emptyset$, and the function $f_1$ itself satisfies
$f_1(0)=f_1(1/2)=0$ by $(A_4)$. 
\end{proof}

\subsection{Even generators}
Throughout this subsection, we assume that $g$ satisfies the assumptions of Theorem~\ref{non_uniq_approx_lattice} and $g(x)=g(-x)$ for every $x\in \R.$ 
Fix $k \in \N$ and define 
$$
\varphi_k(x) = \sum\limits_{n \in \Z} (-1)^n g(x-kn),
$$
$$f_k(x):=\sum_{n\in\Z}g(x-kn).$$
Clearly, both $f_k$ and $\varphi_k$ belong to $V^\infty(g)$. Using the fact that $g$ is even and arguing as in the proof of Lemma~\ref{lem:f_k_properties}, one obtains the following properties of the functions $f_k$ and $\varphi_k$.

\begin{lemma}\label{lem:fk_phi_k} For every $x\in\R$ and $n\in\N$ the following conditions are satisfied:
    \begin{enumerate}
    \item[$(B_1)$] $\varphi_k(x+k) = -\varphi_k(x)$, $\varphi_k(x+2k) = \varphi_k(x),$
    \item[$(B2)$] $\varphi_k(-x) = \varphi_k(x),$
    \item[$(B_3)$] $\varphi_k(-x+k/2) = -\varphi_k(x+k/2),$
    \item[$(B_4)$] $\varphi_k(k/2) = \varphi_k(k/2 + kn) = 0$.
  \item[$(B_5)$] The functions $\varphi_k(x),\varphi_k(x-1),\dots,\varphi_k(x-k+1)$ are linearly independent.
\end{enumerate}
\begin{enumerate}
    \item[$(C_1)$] $f_k(x+k) = f_k(x),$
    \item[$(C_2)$] $f_k(-x) = f_k(x).$
    \item[$(C_3)$] The functions $f_k(x),f_k(x-1),\dots,f_k(x-k+1)$ are linearly independent.
\end{enumerate}
\end{lemma}

\begin{corollary}\label{lem:phi_f_linear_independence} The following statements are true.
\begin{enumerate} 
    \item[$(i)$] If $k$ is odd then the family
    $$\{\varphi_k(\,\cdot\, )\}\cup \{\varphi_k(\,\cdot\,- j)- \varphi_k(\,\cdot\,-k + j)\}_{j=1}^{(k-1)/2}$$
is linearly independent. 
    \item[$(ii)$]
    If $k$ is even then the family
    $$\{f_k(\,\cdot\, ), f_k(\,\cdot\, -k/2)\}\cup \{f_k(\,\cdot\,- j)+ f_k(\,\cdot\,-k + j)\}_{j=1}^{k/2-1}$$
is linearly independent. 
       
\end{enumerate}
   
\end{corollary}

\begin{proof}[Proof of Theorem~\ref{non_uniq_approx_lattice}, part $(ii)$] We distinguish between the cases $k$ even and $k$ odd.

{\bf Case $k\in 2 \N-1$.}
Recall that we are given the set $F_{\frac{k-1}{2}} = \{\pm x_1, \dots, \pm x_{(k-1)/2}\}$ and our goal is to find a function $f$ that vanishes on $\Gamma_k=\left(F_{\frac{k-1}{2}} \cup  \{k/2\}\right) + k \Z.$ 

By Lemma~\ref{lem:phi_f_linear_independence}, part $(i)$, we can solve the system of linear equations and find coefficients $a_0, \dots, a_{(k-1)/2}$ such that the function
\begin{equation*}
    f(x) = a_0\varphi_k(x) + \sum\limits_{j=1}^{(k-1)/2} a_{j} \left( \varphi_k(x-j) - \varphi_k(x-k+j) \right).
\end{equation*}
is non-trivial, belongs to $V^{\infty}(g),$ and  $f(x_j) = 0$ for any $j=1, \dots,{(k-1)/2}.$ 

The relation $(B_1)$ from Lemma~\ref{lem:fk_phi_k} yields 
\begin{equation}\label{eq:f_k_minus}
f(x+k) = - f(x).    
\end{equation}

Using $(B_2)$ and $(B_4)$, we get
\begin{equation}\label{eq:f_k2}
    f(k/2) =  a_0 \varphi_k(k/2) + \sum\limits_{j=1}^{(k-1)/2} a_{j} \left( \varphi_k(k/2-j) - \varphi_k(-k/2+j) \right) = 0.
\end{equation}

In addition, using $(B_1)$ and $(B_2)$, we deduce that 
$$
f(k-x_m) = a_0\varphi_k(k-x_m) + \sum\limits_{j=1}^{(k-1)/2} a_{j} \left( \varphi_k(k-x_m-j) - \varphi_k(-x_m+j) \right)=-f(x_m)=0,
$$
for every $m=1,\dots,(k-1)/2$.
Together with~\eqref{eq:f_k_minus} and~\eqref{eq:f_k2}, this implies that $$f \,\, \text{vanishes on } \Gamma_k=\left(F_{\frac{k-1}{2}} \cup \{k/2\}\right) + k \Z. $$

{\bf Case $k\in 2 \N$.}
By part $(ii)$ of Lemma~\ref{lem:phi_f_linear_independence},  we can find coefficients $a_0, \dots, a_{k/2}$ such that the function
\begin{equation}
    f(x) = a_0 f_k(x) + \sum\limits_{j=1}^{k/2-1} a_{j} \left( f_k(x-j) + f_k(x-k+j) \right)+ a_{k/2}f_k(x-k/2)
\end{equation}
vanishes on $\{x_1, \dots, x_{k/2}\}.$ Clearly, $f\in V^{\infty}(g).$
For every $m=1,\dots,k/2$ we have
$$
f(-x_m)= a_0 f_k(-x_m) + \sum\limits_{j=1}^{k/2-1}  a_{j} \left( f_k(-x_m-j) + f_k(-x_m-k+j) \right)+ a_{k/2}f_k(k/2-x_m),
$$
and using conditions $(C_1)$ and $(C_2)$ from Lemma~\ref{lem:fk_phi_k}, we see that $$f(-x_m)=f(x_m)=0, \quad m=1,\dots,k/2.$$ From $k$-periodicity of the functions $f_k$, we deduce that $f$ vanishes on $\Theta_k=F_{k/2}+k\Z$.
This finishes the proof of Theorem~\ref{non_uniq_approx_lattice}.

 \end{proof}

\section{Proofs of Theorem~\ref{t} and Corollary~\ref{cor:AN_g} }

We prove the statements of Theorem~\ref{t} and Corollary~\ref{cor:AN_g} simultaneously.

$(i)$. Assume that every $g_j$ is an odd function.

Set $$F_j(x):=\sum_{n\in\Z}g_j(x-n), \quad j=1,\dots,N.$$
We leave to the reader to check that every $F_j$ is a $1$-periodic function satisfying 
\begin{equation}\label{eq:F_j_odd_proprties}
    F_j(-x)=-F_j(x), \, x\in\R, \quad \text{and} \quad F_j(0)=F_j\left(\frac{1}{2}\right)=0.
\end{equation}
Furthermore, by mutual linear independence of the windows $g_1, \dots, g_{N},$ we see that the functions $F_j$ are also linearly independent.

Now, given any set $A_N$, one may find a non-trivial linear combination
\begin{equation}\label{eq:F_linear_combination}
F(x)=\sum_{j=1}^N\alpha_j F_j(x), \quad \alpha_j \in \Cc,    
\end{equation}
that vanishes on the set $\{x_1,\dots,x_{N-1}\}$. Every $F_j$ is a $1$-periodic function, and so is $F$.
From~\eqref{eq:F_j_odd_proprties} we get $F(-x)=-F(x)$ and $F(1/2)=0$, from which the statement (i) of Theorem~\ref{t} follows. 

Rewriting~\eqref{eq:F_linear_combination} as
$$F(x)=\sum_{j=1}^N\alpha_j F_j(x) = \sum\limits_{n\in\Z}\sum_{j=1}^N\alpha_j g_j(x-n), \quad \alpha_j \in \Cc,$$
we see that $F\in V^{\infty}(g),$ where $g=\sum\limits_{j=1}^N \alpha_j g_j$ is a non-trivial function, and part $(i)$ of Corollary~\ref{cor:AN_g} also follows. 

$(ii)$. If every $g_j$ is an even function, the proof is similar, where one uses the functions $$ \Phi_j(x):=\sum_{n\in\Z}(-1)^ng_j(x-n), \quad j=1,\dots,N, $$  instead of $F_j$. Note that $\Phi_j$ are $2$-periodic functions that satisfy
$$
\Phi_j(x+1) =- \Phi_j(x), \quad \Phi_j(-x) = \Phi_j(x), \quad \Phi_j\left(\frac{1}{2}\right) = 0.
$$
We leave the details to the reader.

\section{Acknowledgments}
The authors thank Aleksei~Kulikov and Denis~Zelent for several helpful comments that improved the manuscript.

Ilya Zlotnikov was supported by Grant 334466 from the Research Council of Norway and by the Erwin Schrödinger International Institute for Mathematics and Physics (ESI), Vienna, through the Research in Teams project “Gabor Frames in Higher Dimensions.”

\bibliographystyle{abbrv}

\end{document}